\documentclass[11pt]{article}
\usepackage{amsmath, amssymb, amsthm, verbatim,enumerate,bbm}
\usepackage{indentfirst}
\usepackage{circuitikz}
\usepackage{tikz} 
 \usetikzlibrary{calc} 

\date{\today}
\allowdisplaybreaks

\theoremstyle{plain}
\newtheorem{theorem}{Theorem}[section]
\newtheorem{lemma}[theorem]{Lemma}
\newtheorem{claim}[theorem]{Claim}

\newtheorem{corollary}[theorem]{Corollary}

\newtheorem{remark}[theorem]{Remark}
\newtheorem{definition}[theorem]{Definition}

\title{Unions of random trees and applications}

\author{ Austen James
\thanks{Department of Mathematics, Yale University. Email:
austen.james@yale.edu.} \and Matt Larson \thanks{Department of Mathematics, Yale University. Email: matthew.larson@yale.edu} \and Daniel Montealegre \thanks{Department of Mathematics, Yale University. Email: daniel.montealegre@yale.edu}\and Andrew Salmon \thanks{Department of
Mathematics, Yale University. Email: andrew.salmon@yale.edu.} }

\begin{document}
\maketitle
\begin{abstract}
In 1986, Janson showed that the number of edges in the union of $k$ random spanning trees in the complete graph $K_n$ is a shifted Poisson distribution.  Using results from the theory of electrical networks, we provide a new proof of this result, and we obtain an explicit rate of convergence. This rate of convergence allows us to show a new upper tail bound on the number of trees in $G(n,p)$, for $p$ a constant not depending on $n$. The number of edges in the union of $k$ random trees is related to moments of the number of spanning trees in $G(n, p)$.

As an application, we prove the law of the iterated logarithm for the number of spanning trees in $G(n,p)$. More precisely, consider the infinite random graph $G(\mathbb{N}, p)$, with vertex set $\mathbb{N}$ and where each edge appears independently with constant probability $p$. By restricting to $\{1, 2, \dotsc, n\}$, we obtain a series of nested Erd\"{o}s-R\'{e}yni random graphs $G(n,p)$. We show that a scaled version of the number of spanning trees satisfies the law of the iterated logarithm.
\end{abstract}

\section{Introduction}
One of the most basic questions in probability is the following: Given a set $A$ and two randomly chosen subsets $X$ and $Y$, what is the probability that $X$ and $Y$ intersect? Moreover, we can ask about the distribution of the random variable $|X\cup Y|$. This natural question arises in many different contexts. In particular, it has been studied in the context of graphs.\\
\\
Let $G$ be a labeled graph on $n$ vertices, and let $H$ be some unlabeled graph with at most $n$ vertices. Let $\mathcal{S}(H)$ be the set of subgraphs of $G$ which are isomorphic to $H$. If we choose $H_1,H_2\in \mathcal{S}(H)$ independently, uniformly at random, we can ask ``what is the probability that $H_1$ and $H_2$ intersect?" It is clear that if $G=K_n$ and if $H$ is of fixed size, then the probability tends to zero as $n$ tends to infinity. However, this is not necessarily the case when the size of $H$ varies with $n$. \\
\\
 In 1980, Aspvall and Liang solved what they call the ``dinner table problem": if $n$ people are seated at a circular table for two meals, what is the probability that no two people sit next to each other for both meals? This question can be phrased naturally in terms of graph theory: if we independently choose two Hamiltonian cycles in $K_n$ uniformly at random, what is the probability that they are disjoint? Aspvall and Liang showed that this probability approaches $1/e^2$ as $n$ goes to infinity \cite{aspvall1980diner}. The size of the intersections of other types of random subgraphs were studied in \cite{bender1999}.
\\
\\
In 1986, Janson studied the distribution of the number of edges in the union of random trees. Let $t(G)$ denote the set of spanning subtrees in a graph $G$. Let $Po(\lambda)$ denote the Poisson distribution with parameter $\lambda$, whose distribution is given by
\[
\mathbb{P}[Po(\lambda) = t]=e^{-\lambda}\frac{\lambda^t}{t!}.
\]\begin{theorem}[{\cite[Theorem 3]{Janson1986}}]Fix a positive integer $k$. Let $T_1,\ldots,T_k$ be chosen independently, uniformly at random from $t(K_n)$. Define $M_n=k(n-1)-|\cup_i E(T_i)|$. Then
\[
M_n\rightarrow Po(k(k-1)),
\]where the convergence is in distribution.
\end{theorem}
In this paper we extend this results to allow $k$ to grow with $n$. 
\begin{theorem}\label{main thm 1}Let $\alpha \in (0, 1/9)$ be a fixed constant, and let $k=O(n^\alpha)$ be a positive integer. Let $T_1,\ldots,T_k$ be chosen independently, uniformly at random from $t(K_n)$.  Define $M_n:=k(n-1)-|\cup_i E(T_i)|$. Then, 
\[
\sum_{a=0}^{\infty} \left \vert \mathbb{P}[M_n = a] - \mathbb{P}[Po(k(k-1)) = a] \right \vert =o(1).
\]
In particular, if $k$ is a constant independent of $n$, then the total variation distance between $M_n$ and $Po(k(k-1))$ goes $0$. 
\end{theorem}
We have not attempted to maximize the value of $\alpha$ because our methods do not allow $k$ to grow linearly in $n$. In Section 4 we will show how allowing $k$ to grow with $n$ allows us to derive new upper tail estimates for the number of spanning trees in $G(n,m)$ and $G(n,p)$ because Theorem \ref{main thm 1} gives an upper bound on the moments of the number of trees in $G(n,m)$ and $G(n,p)$. It is also worth noting that the method used to prove Theorem \ref{main thm 1} is very different from the one used by Janson. Janson's method proceeded by comparing with a dissassociated set of random variables and uses the moments method. Our method can be easily modified to the case where the trees are drawn from $t(G)$, where $G$ is not the complete graph. 
\\
\\
Lastly, as an application of the upper tail estimates, we will show that the number of spanning trees satisfies a version of the law of the iterated logarithm (LIL). In order to state the result, we first recall a bit of history behind the problem. One of the most important results in probability theory is the central limit theorem (CLT), which states that if $x_1,x_2,\ldots $ is a sequence of independent identically distributed (iid) random variables with mean zero and unit variance, then
\[
\frac{S_n}{\sqrt{n}}\rightarrow N(0,1),
\]where  $S_n:=\sum_{i=1}^n x_i$ and $N(0,1)$ denotes the standard Gaussian. \\
\\
Khinchin \cite{khinchin} and independently Kolmogorov \cite{kolmog}, showed that under the same conditions one has
$$ \mathbb{P}\left[ \limsup_{n\rightarrow \infty} \frac{ S_n }{\sqrt{n} \sqrt {2 \log \log n }} =1 \right] =1, $$
which has been referred to as the Law of the Iterated Logarithm. 
\\
\\
There has been much work to extend CLT to the case where one allows dependence among $x_i$. In particular, it has been studied for the graph count case. Let $\tilde{\mathcal{A}}$ be a set of unlabeled graphs on at most $n$ vertices, and denote by $\mathcal{A}$ the set of copies of $\tilde{\mathcal{A}}$ in $K_n$. Then, we can define:
\[
X_n = \sum_{H\in \mathcal{A}} \mathbb{I}_{H\in G}
\]where $\mathbb{I}_{H\in G}$ is the indicator random variable for the event $H\in G$, and $G$ is some random graph (it can be sampled from $G(n,m)$, $G(n,p)$, or any other random graph model). Then $X_n$ is precisely the number the of copies of $\tilde{\mathcal{A}}$ in some random graph. For example, if we let $\tilde{\mathcal{A}}=\{3-$cycle$\}$, then $X_n$ is precisely the number of triangles in a random graph $G$. 
\\
\\
Many papers have studied graph counts. In particular, Ruci{\'n}ski found necessary and sufficient conditions for the number of copies of a fixed graph to be normally distributed \cite{rucinski1988small}. 
For larger graphs, Riordan found probabilities $p$ that are close to the best possible for a cube and square lattice to appear in a $G(n, p)$ with probability tending to $1$ \cite{riordan}.
Janson showed in \cite{janson} that if we let $G\sim G(n,m)$, then the (normalized) number of Hamilton cycles, spanning trees, and perfect matchings tend towards the standard normal distribution \cite[Theorem 2]{janson}. However, if $G\sim G(n,p)$, then this is not the case.
\begin{theorem}[{\cite[Theorem 4]{janson}}] Fix a constant $p < 1$. Let $X_n$ be the random variable that counts number of spanning trees, perfect matchings, or Hamilton cycles in $G(n,p)$. Let $p(n)\rightarrow p$. If $\liminf n^{1/2} p(n)>0$, then 
$$
p(n)^{1/2}\left(\log X_n-\log\mathbb{E}X_n+\frac{1-p(n)}{c p(n)}\right)\rightarrow N\left(0,\frac{2(1-p)}{c}\right)
$$where $c=1$ in the case of spanning trees and Hamilton cycles, and $c=4$ in the case of perfect matchings. 
\end{theorem}

Although CLT has been widely studied, this is not the case for LIL. In \cite{2016arXiv160708865F}, Ferber, Montealegre, and Vu showed that LIL holds for the number of copies of a graph with fixed size \cite[Theorem 1.3]{2016arXiv160708865F}. Moreover, they showed that a version of LIL holds for the case of Hamilton cycles \cite[Theorem 1.4]{2016arXiv160708865F}. In this paper we show a version of the LIL for  spanning trees.
\begin{theorem}\label{maintheorem3}
Let $0<p<1$ be a constant. Let $X_n$ be the number of spanning trees in $G(n,p)$, coupled by forming $G(\mathbb{N}, p)$ and then restricting to $[n]$. Then, \begin{equation}
\mathbb{P}\left[\limsup_{n \to \infty}\frac{\log X_n -\mu_n}{\sigma \sqrt{2 \log \log n}} = 1\right] = 1,
\end{equation}
where $\mu_n = \log(p^{n-1} n^{n-2})$ and $\sigma = \sqrt{\frac{2 (1-p)}{p}}$.
\end{theorem}

The organization of the paper is as follows. In Section 2, we present some notation and results that will be used throughout the paper. In Section 3, we prove Theorem \ref{main thm 1}. In Section 4, we derive new upper tail estimates for the number of spanning trees, which might be of independent interest. Section 5 contains the proof of Theorem \ref{maintheorem3}. Lastly, Section 6 contains some calculations which we have omitted in some of the earlier sections for sake of clarity.

\section{Background and notation}
Let $G(\mathbb{N},p)$ be the random graph on vertex set $\mathbb{N}$ where any two vertices are joined independently at random with probability $p$. Let $G(n,p)$ denote the subgraph induced by the first $n$ vertices. Throughout this paper, we will only consider the case where $p$ is a fixed constant. Let $G(n,m)$ be the random graph model on $n$ vertices formed by choosing a set of $m$ distinct edges uniformly at random. 
\\
\\
Let $X_n$ the number of spanning trees in $G(n,p)$, and let $X_{n,m}$ denote the number of trees in $G(n,m)$. By Cayley's formula, $N_T:=n^{n-2}$ will denote the number of spanning trees in $K_n$. 
\\
\\
We shall repeatedly use the following well-known theorem.

\begin{theorem}[Borel-Cantelli Lemma]\label{BorelCantelli} Let $(A_i)_{i=1}^{\infty}$ be a sequence of events. If $\sum_{i=1}^{\infty} \mathbb{P}[A_i] < \infty$, then \begin{equation*}
\mathbb{P}\left[A_i \text{ holds for infinitely many }i \right]=0
\end{equation*}
\end{theorem}

We will use several results from the theory of electrical networks in our proof. See \cite{LP} for a detailed introduction to the theory of electrical networks. For the sake of completeness, we will briefly summarize some basic theorems that will be used. An electrical network is a multigraph with weighted edges $R : E \to \mathbb{R}_{\ge 0}$, called resistances. Every graph produces an electrical network by assigning each edge a resistance of $1$.
The graphs we consider are not directed, but we can view them as reversible Markov chains with transition probabilities defined so that they are inversely proportional to resistances $p(x, y) = p(y, x) \sim 1/R_{xy}$ for each edge $xy$.
We will also consider a potential function, or voltage function $v: V \to \mathbb{R}_{\ge 0}$ on our graphs.
Voltage functions that we consider will fix values for two vertices, viewed as a source and sink. Outside of these distinguished vertices, the voltage function will be a harmonic function, which enforces at each vertex $x$ the averaging property 
\begin{equation*}
    v(x) = \sum_{xy \in E} p(x,y) v(y).
\end{equation*}
For a finite connected network, $v$ is completely determined by the harmonic property once the voltage is fixed for any two vertices. A voltage function defines a current function $i: E \mapsto \mathbb{R}$ that assigns a current, the amount of electricity flowing through a resistor, to a directed edge. This current is defined by the Ohm's law.

\begin{definition}[Ohm's Law] Let $ab$ be an edge in $H$. Let $v(a) - v(b) $ be the voltage difference across $ab$, and let $R_{ab}$ the resistance of $ab$. Then  \begin{equation*}
i(ab) = \frac{v(a) - v(b)}{R_{ab}}
\end{equation*}
\end{definition}

Once we have defined a source, sink, and voltage function, we can consider network reductions, which modify the electrical network but leave the voltages at certain vertices unchanged.
The two network reductions we will use are the series law and parallel law. We say two resistors are in series if they are arranged in a chain.
A property of resistors is that the resistance of resistors in circuits can be added together, so the two systems below are equivalent for the purposes of computing $v(a)$ and $v(b)$. 

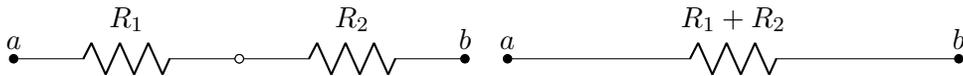
\begin{center} \begin{circuitikz} \draw
(0,0) to[R=$R_1$, o-o] (3,0)
(3,0) to[R=$R_2$, o-o] (6,0)
;
\node [above] at (0,0) {$a$};
\draw[fill] (0,0) circle [radius=0.05];
\draw[fill] (6,0) circle [radius=0.05];
\node [above] at (6,0) {$b$};
\end{circuitikz}
\begin{circuitikz} \draw
(0,0) to[R=$R_1+R_2$, o-o] (6,0)
;
\node [above] at (0,0) {$a$};
\draw[fill] (0,0) circle [radius=0.05];
\draw[fill] (6,0) circle [radius=0.05];
\node [above] at (6,0) {$b$};
\end{circuitikz}\end{center}

Two resistors are in parallel if they both have the same endpoints. 
Resistors in parallel can be combined by taking the harmonic mean of the resistances.

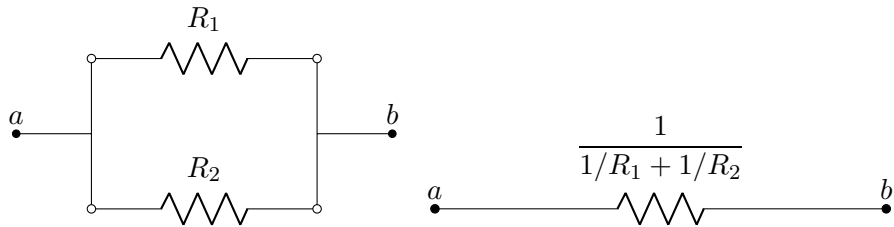
\begin{center} \begin{circuitikz} \draw
(0,0) -- (1,0)
(1,0) -- (1,1)
(1,1) to[R=$R_1$, o-o] (4,1)
(4,1) -- (4,0)
(4,0) -- (5,0)
(1,0) -- (1,-1)
(1,-1) to[R=$R_2$, o-o] (4,-1)
(4,-1) -- (4,0)
;
\node [above] at (0,0) {$a$};
\draw[fill] (0,0) circle [radius=0.05];
\draw[fill] (5,0) circle [radius=0.05];
\node [above] at (5,0) {$b$};
\end{circuitikz}
\begin{circuitikz} \draw
(0,0) to[R= $\dfrac{1}{1/R_1 + 1/R_2}$, o-o] (6,0)
;
\node [above] at (0,0) {$a$};
\draw[fill] (0,0) circle [radius=0.05];
\draw[fill] (6,0) circle [radius=0.05];
\node [above] at (6,0) {$b$};
\end{circuitikz}
\end{center}

The following theorem, originally due to Kirchhoff (see $\cite[\text{p. 105}]{LP}$) establishes a connection between electrical networks and trees. Intuitively, the probability that a random spanning tree uses a given edge $ab$ depends on how many other paths there are from $a$ to $b$. If there are few other paths from $a$ to $b$, then the current from $a$ to $b$ will be large, and the probability that a random spanning tree contains $ab$ will also be large. This motivates the following theorem. 
\begin{theorem} \label{kirch} Let $ab$ be an edge in $H$ where $v(b)=0$ and $v(a)$ is such that $\sum_{ea \in E} i(ea) = 1$. Suppose we choose a spanning tree, $T$, uniformly at random from a graph $H$. Then
\begin{equation*} \mathbb{P}\left[ab \in T \right] = i(ab), \end{equation*} 
\end{theorem}
We also have the following theorem, which says that the events that a random spanning tree contains a given edge are negatively correlated.
\begin{theorem}[{\cite[Theorem 4.5]{LP}}]\label{prod} Let $e_1, \dots, e_k$ be edges in $H$ and choose a spanning tree $T \subseteq H$ uniformly at random.  Then
\begin{equation*} 
\mathbb{P}\left[e_1, e_2, \dotsc, e_k \in T \right] \le \prod_{i=1}^k \mathbb{P}\left[e_i \in T \right].
\end{equation*}
\end{theorem}
Note also the following theorem.
\begin{theorem}[Rayleigh Monotonicity Law] Let $G$, $G'$ be graphs on the same vertex set with the same voltage fixed at the source and sink. Suppose that $G' \subseteq G$. Then for every edge $e \in G'$, \begin{equation*}
i_{G'}(e) \ge i_{G}(e).
\end{equation*}
\end{theorem}
This, together with Ohm's law, is equivalent to the assertion that adding a resistor to a network cannot decrease the current through any edge.  

We also use the following bound on the number of subtrees.
\begin{theorem}[\cite{Grimmett}]\label{grimmett} Let $H$ be a graph with $n$ vertices and $m$ edges. Then,
\begin{equation*}
|t(H)| \le \frac{1}{n} \left(\frac{2m}{n-1}\right)^{n-1}.
\end{equation*}
\end{theorem}
We will also need the following theorem, which shows that trees chosen randomly from $t(K_n)$  have small maximum degree with high probability. Let $\Delta(G)$ denote the maximum degree of a graph $G$. 
\begin{theorem}[{\cite[Theorem 1]{moon}}]\label{moon} Let $T$ be chosen uniformly at random from $t(K_n)$. Then
$$\mathbb{P}[\Delta(T) > \ell] \le \frac{n}{\ell!}.$$
\end{theorem} 

\section{Proof of Theorem \ref{main thm 1}}

 Let $k$ be an integer. Let $T_1,\ldots, T_k$ be trees chosen uniformly at random from $t(K_n)$. Let $M_n:=k(n-1)-|\cup_i E(T_i)|$. In order to show Theorem \ref{main thm 1} we will show the following two claims.
\begin{claim}\label{claim 1}We have that
\begin{equation}\label{first upper bound}
\mathbb{P}[M_n=a]\leq \frac{(k(k-1))^a}{a!}.
\end{equation}
\end{claim}
\begin{claim}\label{claim 2}
Moreover, if we know that $k=O(n^\alpha)$ and $a\leq n^{3\alpha} $, we can improve the above upper bound to
\begin{equation}\label{second upper bound}
\mathbb{P}[M_n=a]\leq \left(1+o(1)\right)\mathbb{P}[Po(k(k-1))=a].
\end{equation}
\end{claim}
While the claims only show upper bounds, a straightforward calculation yields the desired asymptotic results:
\begin{proof}[Proof of Theorem \ref{main thm 1}]
Set $\mathcal N:=\{a\mid \mathbb{P}[M_n = a] > \mathbb{P}[Po(k(k-1))=a]\}$.
As probabilities must have total sum $1$, we know that
\begin{equation}  \sum_{a=0}^{\infty} \vert \mathbb{P}[M_n = a] - \mathbb{P}[Po(k(k-1))=a] \vert = 2\sum_{a\in \mathcal{N}} \mathbb{P}[M_n = a] - \mathbb{P}[Po(k(k-1))=a] 
\end{equation}
We split the above sum into two parts
\[
S_1= 2\sum_{a\in\mathcal N_{\leq n^{3\alpha}}} \mathbb{P}[M_n = a] - \mathbb{P}[Po(k(k-1))=a],
\]
\[
S_2= 2\sum_{a\in\mathcal N_{> n^{3\alpha}}} \mathbb{P}[M_n = a] - \mathbb{P}[Po(k(k-1))=a],
\]where $\mathcal N_{\leq n^{3\alpha}}:=\{a\mid \mathbb{P}[M_n = a] > \mathbb{P}[Po(k(k-1))=a]$ and $a\leq n^{3\alpha}\}$ and $\mathcal N_{> n^{3\alpha}}:=\{a\mid \mathbb{P}[M_n = a] > \mathbb{P}[Po(k(k-1))=a]$ and $a> n^{3\alpha}\}$.
\\
\\
Using Claim \ref{claim 1}, we can upper bound $S_1$ by 
\begin{align*}
S_1\leq 2\sum_{a\in\mathcal N_{\leq n^{3\alpha}}}o(1)\mathbb{P}[Po(k(k-1))=a]=o(1).
\end{align*}To upper bound $S_2$ we use Claim \ref{claim 2} to obtain
\begin{align*}
S_2\leq 2\sum_{a\in\mathcal N_{> n^{3\alpha}}}\frac{(k(k-1))^a}{a!}(1-e^{-k(k-1)})\leq 2\sum_{a=n^{3\alpha}}^\infty\frac{(k(k-1))^a}{a!}=o(1),
\end{align*}where the last equality holds because $k=O(n^\alpha)$ and $a>n^{3\alpha}$. The upper bounds on $S_1$ and $S_2$ imply our result. 
\end{proof}
Now we show the desired claims:
\begin{proof}[Proof of Claim \ref{claim 1}]We wish to upper bound the number of $k$-tuples $(T_1,\ldots,T_k)$ such that their union contains exactly $k(n-1)-a$ edges. To this end, let $(\ell_2,\ldots,\ell_k)$  a partition of $a$ (that is, $\sum_i \ell_i=a$). We run the following algorithm.
\begin{enumerate}
\item First, choose $T_1$. 
\item For $i=2,3,\ldots, k$: 
\begin{enumerate}
\item Having chosen $T_1,\ldots, T_{i-1}$  choose $\ell_i$ edges in $T_1\cup\ldots\cup T_{i-1}$. Call this set of edges $S_i$. 
\item Complete $S_i$ into a tree without using any other edges in  $T_1\cup\ldots\cup T_{i-1}$. Call this resulting tree $E(T_i)$. If $S_i$ cannot be completed into a tree, then return nothing. 
\end{enumerate}
\item Return $(T_1,\ldots, T_k)$. 
\end{enumerate}
We now upper bound the number of outputs we can get after running this algorithm. Clearly, we have $N_T$ ways to perform step 1. Also, the number of ways to perform step 2a (at iteration $i$) is given by
\[
{ | T_1\cup\ldots\cup T_{i-1}|\choose \ell_i}.
\]We have a clear upper bound given by
\begin{equation}\label{iteration i}
{ | T_1\cup\ldots\cup T_{i-1}|\choose \ell_i} \leq \frac{((i-1)(n-1))^{\ell_i}}{\ell_i!}\leq \frac{n^{\ell_i}(i-1)^{\ell_i}}{\ell_i!}.
\end{equation}
In order to upper bound step 2b (at iteration $i$), we need to upper bound the number of trees that contain the set $S_i$. First of all, note that for an edge $e \in E(K_n)$ and $T\in t (K_n)$ is chosen at random, then $\mathbb{P}[e\in T]=2/n$. Therefore, Theorem \ref{prod} gives that the number of trees that contain $S_i$ is upper bounded by
\begin{equation}\label{step 2b}
\frac{N_T2^{\ell_i}}{n^{\ell_i}}.
\end{equation}
Combining equations \eqref{iteration i} and \eqref{step 2b}, together with the upper bound on step 1, we obtain an upper bound on the number of outputs of
\begin{equation}\label{bad upper bound}
N_T^k\prod_{i=2}^k \frac{(2(i-1))^{\ell_i}}{\ell_i!}.
\end{equation}Now we add over all possible partitions of $a$ to obtain 
\begin{align*}
\sum_{\ell_2+\ldots+\ell_k=a} N_T^k\prod_{i=2}^k \frac{(2(i-1))^{\ell_i}}{\ell_i!}&=\frac{N_T^k}{a!}\sum_{\ell_2+\ldots+\ell_k=a}a!\prod_{i=2}^k \frac{(2(i-1))^{\ell_i}}{\ell_i!} \nonumber\\&=N_T^k\frac{(2+4+\ldots+2(k-1))^a}{a!}\nonumber\\
&=\frac{N_T^k (k(k-1))^a}{a!},
\end{align*}where the second equality is due to the multinomial theorem. Dividing by $N_T^k$ gives \eqref{first upper bound}.
\end{proof}
Before proving Claim \ref{claim 2}, we prove a lemma.
\begin{lemma} \label{treeprob}
Let $G$ be a graph with minimum degree $\delta = n - k$. Let $e \in E(G)$ be any edge. Choose a tree $T$ uniformly at random from $t(G)$. Then
$$\mathbb{P}[e \in T] \le \frac{2}{n - 2k + 2}.$$
\end{lemma}
\begin{proof}
Consider some edge $e = ab \in G$. There are at least $n - 2k$ paths of length $2$ from $a$ to $b$. Let $G'$ be the electrical network that consists of the edge $ab$ and every path of length $2$ from $a$ to $b$, each edge having resistance $1$, and with $v(b) = 0$ and $v(a)$ such that $\sum_{ea \in E(G')}i(ea) = 1$. We shall find an upper bound on the probability that $ab$ is in $T$ by bounding $i_{G'}(ab)$ and using Theorem \ref{kirch}.

\begin{center} \begin{circuitikz}
\draw
(0,0) -- (4,0)
(0,0) -- (2,1)
(2,1) -- (4,0)
(0,0) -- (2,3)
(2,3) -- (4,0)
(0,0) -- (2,1.5)
(2,1.5) -- (4,0)
;
\node [above] at (0,0) {$a$};
\draw[fill] (0,0) circle [radius=0.05];
\draw[fill] (4,0) circle [radius=0.05];
\node [above] at (4,0) {$b$};
\draw[fill] (2,1) circle [radius=0.05];
\draw[fill] (2,1.5) circle [radius=0.05];
\draw[fill] (2,3) circle [radius=0.05];
\draw[fill] (2,1.8) circle [radius=0.025];
\draw[fill] (2,2.1) circle [radius=0.025];
\draw[fill] (2,2.4) circle [radius=0.025];
\draw[fill] (2,2.7) circle [radius=0.025];

\end{circuitikz}
\end{center}
We may convert each path of length $2$ into a single edge with resistance $2$.

\begin{center}

\begin{tikzpicture}
\node[circle,fill=black,inner sep=1pt,draw] (a) at (0,0) {};
\node[circle,fill=black,inner sep=1pt,draw] (b) at (4,0) {};
\draw[] (a) edge[bend left=100] (b);
\draw[] (a) edge[bend left=10] (b);
\draw[] (a) edge (b);
\node [below] at (0,0) {$a$};
\node [below] at (4,0) {$b$};
\node [below] at (2,0) {$R=1$};
\node [above] at (3,0.2) {$R=2$};
\node [above] at (3,1.2) {$R=2$};
\draw[fill] (2,0.8) circle [radius=0.025];
\draw[fill] (2,0.4) circle [radius=0.025];
\draw[fill] (2,0.6) circle [radius=0.025];
\draw[fill] (2,1) circle [radius=0.025];
\end{tikzpicture}

\end{center}
As these resistors are in parallel, for the purposes of computing $v(a)$ $G'$ is equivalent to a single edge with resistance
\begin{equation*} \label{resistance}
\frac{1}{R_{ab}} \ge (n - 2k)\frac{1}{2} + 1 = \frac{n - 2k + 2}{2}.
\end{equation*}
Thus $R_{ab} \ge \frac{2}{n - 2k + 2}$.
Using Ohm's law, we see that
\begin{equation*} \label{voltage}
v(a) - v(b) = v(a) \le \frac{2}{n - 2k + 2},
\end{equation*} 
as the current is $1$ by construction. Using Ohm's law on the single edge $ab$, we see that
\begin{equation*} 
\frac{2}{n - 2k + 2} \ge i_{G'}(ab). 
\end{equation*}
By construction, $G'$ can be embedded into $G$, so, because of Rayleigh Monotonicity Law, $i_{G'} \ge i_{G}$. Therefore, if $T$ is chosen uniformly at random from $t(G)$, then by Theorem \ref{kirch},
\begin{equation*}
\mathbb{P}\left[ e \in T \right ] \le \frac{2}{n - 2k + 2}.
\end{equation*}
\end{proof}

\begin{proof}[Proof of Claim \ref{claim 2}]Let 
\[
\mathcal{M}(a)=\{(T_1,\ldots,T_k) : \enskip |\cup_i E(T_i)|=(k-1)(n-1)-a\}.
\]First, we partition $\mathcal{M}(a)$ into two sets.  Let
\[
\mathcal{N}_1(a)=\{(T_1,\ldots,T_k)\in \mathcal{M}(a): \enskip \Delta(T_i)\leq n^{4\alpha} \quad \forall i\},
\]and let $\mathcal{N}_2(a)=\mathcal{M}(a)\backslash \mathcal{N}_1(a)$. We wish to upper bound $|\mathcal{M}(a)|$. Because most trees have small maximum degree, $\mathcal{N}_2(a)$ will be negligible. We have good control over $\mathcal{N}_1(a)$ because the graph obtained by deleting from the complete graph some trees of small maximum degree has large minimum degree, so no edge will be included in too many spanning trees. We will prove that $|\mathcal{N}_2(a)|=o(|\mathcal{N}_1(a)|)$ and that $|\mathcal{N}_1(a)|$ satisfies the desired upper bound. 
\\
\\ 
Let $(\ell_2,\ldots,\ell_k)$ be a partition of $a$. We run the following algorithm.
\begin{enumerate}
\item Choose $T_1$ such that $\Delta(T_1)\leq n^{4\alpha}$. 
\item For $i=2,3,\ldots, k$: 
\begin{enumerate}
\item Let $U_i=T_1\cup\ldots\cup T_{i-1}$ and choose $\ell_i$ edges in $U_i$. Call this set of edges $S_i$. 
\item Let $G_{S_i}=K_n\backslash(U_i\backslash S_i)$. Complete $S_i$ into a tree (in $G_{S_i}$) with max degree at most $n^{4\alpha}$. Call it $T_i$. If no such tree exists, return nothing. 
\end{enumerate}
\item Return $(T_1,\ldots,T_k)$.
\end{enumerate}
Now we upper bound the number of outputs the above algorithm can produce. Step 1 can be upper bounded by $N_T$. Step 2a (iteration $i$) can be performed in at most
\begin{equation}\label{factor 1}
{|U_i| \choose \ell_i}\leq \frac{|U_i|^{\ell_i}}{\ell_i!}\leq\frac{((i-1)n)^{\ell_i}}{\ell_i!}
\end{equation} ways. Let $T$ be chosen uniformly at random from $t(G_{S_i})$. Note that an upper bound on Step 2b (iteration $i$) is given by the total number of trees in $G_{S_i}$ containing $S_i$, which is given by
\[
\mathbb{P}[S_i\subseteq T]\cdot |t(G_{S_i})|.
\]We upper bound each factor individually. For the latter factor we use Theorem \ref{grimmett}. Using that $\vert E(G_{S_i}) \vert \le n(n-1)/2 - (n-1)(i-1)$ and that $(1 + x/n)^{n} = e^{x + O(x^2/n)}$, we see that 
\begin{equation} \label{tau} \begin{split} 
|t(G_{S_i})| &\le \frac{1}{n} (n + 2(1-i))^{n-1} \\
&= n^{n-2} (1 + 2(1-i)/n)^{n-1} \\
& = N_T e^{2-2i+O(n^{2\alpha}/n)}. 
\end{split}
\end{equation}By Theorem \ref{prod}, we have that
\begin{equation}\label{Si probability}
\mathbb{P}[S_i\subseteq T]\leq \left( \max_{e\in G_{S_i}}\quad \mathbb{P}[e\in G_{S_i}]\right)^{\ell_i}.
\end{equation}
By construction, we have that $\Delta(T_t) < n^{4\alpha}$ for all $t$, so $\delta(G_{S_i}) \ge n - (i-1)n^{4\alpha}$, where $\delta(H)$ is the minimum degree of a vertex in $H$. Let $e \in E(G_{S_i})$. By Lemma \ref{treeprob}, we see that for a tree chosen uniformly at random from $G_{S_i}$, 
\begin{equation} \label{probability}
\mathbb{P}[e \in T] \le \frac{2}{n - 2(i - 1)n^{4 \alpha} + 2} = \frac{2}{n - O(n^{5 \alpha})}.
\end{equation}
Using this on \eqref{Si probability}, we obtain
\begin{equation}\label{factor 2}
\mathbb{P}[S_i\subseteq T]\leq \left(\frac{2}{n - O(n^{5\alpha})}\right)^{\ell_i}.
\end{equation}
Putting together equations \eqref{factor 1}, \eqref{tau}, and \eqref{factor 2}, we obtain an upper bound on the number of ways step 2 (iteration $i$) can be performed of
\begin{align}
&\frac{((i-1)n)^{\ell_i}}{\ell_i!}\left(N_T e^{2-2i+O(n^{2\alpha}/n)}\right)\left(\frac{2}{n - O(n^{5\alpha})}\right)^{\ell_i}\nonumber\\
&=\frac{N_T(2(i-1))^{\ell_i}}{e^{2(i-1)}\ell_i!}\left(1+O\left(\frac{n^{8\alpha}}{n}\right)\right).
\end{align}Hence, the number of ways to perform step 2 is at most
\[
N_T^k\prod_{i=2}^k \frac{(2(i-1))^{\ell_i}}{e^{2(i-1)}\ell_i!}\left(1+O\left(\frac{n^{8\alpha}}{n}\right)\right)=N_T^ke^{-k(k-1)}\left(1+O\left(\frac{n^{9\alpha}}{n}\right)\right)\prod_{i=2}^k\frac{(2(i-1))^{\ell_i}}{\ell_i!}.
\]Now we add over all possible partitions of $a$ to obtain
\[
N_T^ke^{-k(k-1)}\left(1+O\left(\frac{n^{9\alpha}}{n}\right)\right)\sum_{\ell_2+\ldots+\ell_k=a}\prod_{i=2}^k\frac{(2(i-1))^{\ell_i}}{\ell_i!}.
\]Applying the multinomial theorem, we obtain the desired upper bound
\begin{equation}\label{N1}
|\mathcal{N}_1(a)|\leq \frac{N_T^ke^{-k(k-1)}(k(k-1))^a}{a!}\left(1+O\left(\frac{n^{9\alpha}}{n}\right)\right).
\end{equation}Now we upper bound $|\mathcal{N}_2(a)|$. 
From Theorem \ref{moon}, we see that
\[
\mathbb{P}\left[ \Delta(T) > n^{4\alpha} \right] \le \frac{n}{(n^{4\alpha})!}.
\]Hence, the number of $k$-tuples that have at least one tree with max degree more than $n^{4\alpha}$ is upper bounded by
\[
k\cdot\frac{n}{(n^{4\alpha} )!}\cdot N_T^k.
\]Because the above is an upper bound for $|\mathcal{N}_2(a)|$, using a straight forward calculation (see appendix) we obtain
\begin{equation}\label{N2}
|\mathcal{N}_2(a)|\leq N_T^k\frac{e^{-k(k-1)}(k(k-1))^a}{a!}\cdot O\left(\frac{n^{9\alpha}}{n}\right).
\end{equation}Because $|\mathcal{M}(a)|=|\mathcal{N}_1(a)|+|\mathcal{N}_2(a)|$, using equations \eqref{N1} and \eqref{N2} we obtain 
\[
|\mathcal{M}(a)|\leq N_T^k \frac{e^{-k(k-1)}(k(k-1))^a}{a!}\left(1+O\left(\frac{n^{9\alpha}}{n}\right)\right).
\]As $\alpha < 1/9$, dividing by $N_T^k$ gives the desired claim.

\end{proof}
\section{Upper tail estimates}In this section we present some new upper tail estimates that might be of independent interest. Let $X_{n, m}$ denote the number of spanning trees in the random graph $G(n, m)$. Our main goal in this section is to prove the following lemma. 

\begin{lemma}\label{concentrationlemma} Let $0 < \delta < 1/2$ be a constant, and let $0 < \alpha < 1/9$. There is a constant C depending on $\delta$ such that for any $\delta n^2 \le m \le (1 - \delta) n^2$, and $k=O(n^{\alpha})$, we have
\begin{equation*}
\mathbb{E}X_{n,m}^k \le C^k (\mathbb{E}X_{n,m})^k. 
\end{equation*}
\end{lemma}
Using Markov's Inequality, we have that \begin{equation*}
\mathbb{P}\left[X_{n,m} \ge K\mathbb{E}X_{n,m} \right] = \mathbb{P}\left[X_{n,m}^k \ge (K\mathbb{E}X_{n,m})^k \right] \le \left(\frac{C}{K}\right)^k.
\end{equation*}
letting $k=n^{\alpha}$ with $\alpha < 1/9$ and $K=Ce^t$ we obtain the following lemma, which will be used in the proof of the upper bound for the LIL. 
\begin{lemma}\label{XKBound}Let $0<\delta <1/2$ and $0 < \alpha < 1/9$ be constants, and let $t\geq 0$ be a fixed integer. Then there exists a constant $K$ such that for any $\delta n^2\leq m\leq (1-\delta)n^2$ we have:
\[
\mathbb{P}[X_{n,m} \ge K\mathbb{E}X_{n,m}]\leq \exp(t n^{-\alpha}).
\]

In particular, for $k = \log n$, we obtain:
\[
\mathbb{P}[X_{n,m} \ge K\mathbb{E}X_{n,m}]\leq n^{-t}.
\]
\end{lemma}

\begin{remark}
We will only need the upper bound of $n^{-t}$ for fixed integer $t \ge 0$ and will not need the subexponential bound in the remainder of the paper.

Our techniques would not allow the subexponential bound to be made exponential as that would require us to allow $k$ to grow linearly with $n$.
\end{remark}

Before we proceed with the proof of Lemma \ref{concentrationlemma}, we need a little bit of background: For any fixed graph $J$ with $j$ edges, the probability that $J$ appears in $G(n,m)$ is precisely \begin{equation*}
\frac{\displaystyle\binom{\binom{n}{2} - j}{m- j}}{\displaystyle\binom{\binom{n}{2}}{m}}  = \frac{(m)_j}{\left(\binom{n}{2}\right)_j},
\end{equation*}
where $(N)_{\ell} = N(N-1)\dotsb(N-\ell+1)$.

For each $T \in t(K_n)$, let $X_T$ denote the event that ``$T$ appears in $G(n,m)$". Then $X_{n,m} = \sum_{T \in t(K_n)}X_{T}$. Therefore, \begin{equation*}
\mathbb{E}X_T = \frac{(m)_{n-1}}{\left(\binom{n}{2}\right)_{n-1}.}
\end{equation*}
 Thus, by linearity, 
\begin{equation*}
\mathbb{E}X_{n,m} = N_T\frac{(m)_{n-1}}{\left(\binom{n}{2}\right)_{n-1}},.
\end{equation*}

We shall repeatedly use the following estimate, which is proved in the appendix. Let $N, \ell$ such that $\ell = o(N^{2/3})$. Then \begin{equation} \label{fallingfactorial}
(N)_{\ell} = N^{\ell}\exp\left( -\frac{\ell(\ell-1)}{2N} + O(\ell^3/N^2) \right).
\end{equation}

Let $J = \cup_{i=1}^{k} T_i$, where $(T_1, T_2, \dotsc, T_k)$ is a $k$-tuple of elements of $t(K_n)$.  Let $M(a)$ be the number of $k$-tuples of elements of $t(K_n)$ such that $\vert \cup_{i=1}^k E(T_i)\vert = k(n-1)-a$.  Since $X_J = X_{T_1} \cdots X_{T_k}$, we see that \begin{equation}\label{kmoment}
\mathbb{E}X_{n,m}^k = \sum_{(T_1,\dotsc, T_k) \in t(K_n)^k} \mathbb{E}[X_{T_1} \dotsb X_{T_k}] = \sum_{a=0}^{(k-1)(n-1)}M(a) \frac{(m)_{(n-1) k-a}}{(\binom{n}{2})_{(n-1) k-a}}.
\end{equation}

Let $p_m = \frac{m}{\binom{n}{2}}$. By $(\ref{fallingfactorial})$, for all $a$, \begin{equation}\label{fallingpart}
\frac{(m)_{(n-1) k-a}}{\left(\binom{n}{2}\right)_{(n-1) k-a}} \le p_m^{(n-1) k-a}\exp\left(\frac{-k^2(1-p_m)}{p_m} + O(n^{-2/3})\right).
\end{equation} 

In particular,  letting $k=1$ and $a = 0$ gives
\begin{equation}\label{EXn}
\mathbb{E}X_{n,m} = N_Tp_m^{n-1}\exp\left(-\frac{1-p_m}{p_m} + O(n^{-2/3}) \right).
\end{equation}

Now we carry on with the proof of the upper tail estimate.

\begin{proof}[Proof of Lemma \ref{concentrationlemma}]
Recall from (\ref{kmoment}) that \begin{equation*}
\mathbb{E}X_{n,m}^k = \sum_{a=0}^{(k-1)(n-1)}M(a) \frac{(m)_{(n-1) k-a}}{(\binom{n}{2})_{(n-1) k-a}}.
\end{equation*}
We split the RHS into the sum up to $T = \lceil 2 k^2e/p_m \rceil$ and the rest of the sum. For ease of notation we assume that $k$ (and thus $T$) tends to infinity, but the proof can be easily modified if $k$ is bounded. Note that $T \le O(n^{2\alpha})$, so we can apply Claim \ref{claim 2} to obtain
\begin{equation*}
S_{\le T} := \sum_{a=0}^T M(a) \frac{(m)_{(n-1) k-a}}{(\binom{n}{2})_{(n-1) k-a}} \le 2\frac{p_m^{(n-1) k} N_T^k}{e^{k(k-1)}}\exp\left(-\frac{k^2(1-p_m)}{p_m}+ O(n^{-2/3}) \right) \sum_{a=0}^T \frac{( k(k-1))^a}{a!}p_m^{-a}.
\end{equation*}
As \begin{equation*}
\sum_{a=0}^T \frac{( k(k-1))^a}{a!}p_m^{-a} \le \sum_{a=0}^{\infty} \frac{(k(k-1))^a}{a!}p_m^{-a} = e^{ k(k-1)/p_m},
\end{equation*}
we see that
\begin{equation} \label{s1}
S_{\le T} \le 2\frac{p_m^{(n-1) k} N_T^k}{e^{k(k-1)}}\exp\left(-\frac{k^2(1-p_m)}{p_m}+ O(n^{-2/3}) \right) e^{k(k-1)/p_m} \le C_1^k p_m^{(n-1) k} N_T^k
\end{equation}
for an appropriate constant $C_1$. Using Claim \ref{claim 1}, we get an upper bound on $S_2$
\begin{equation*}
S_{> T} := \sum_{a>T} M(a) \frac{(m)_{(n-1) k-a}}{\left(\binom{n}{2}\right)_{(n-1) k-a}} \le p_m^{(n-1) k} N_T^k\exp\left(-\frac{k^2(1-p_m)}{p_m}+ O(n^{-2/3}) \right) \sum_{a>T} \frac{( k(k-1))^a}{a!}p_m^{-a}.
\end{equation*}
Using Stirling's approximation, we have that \begin{equation*}
\sum_{a>T}\frac{( k(k-1))^a}{a!p_m^a} \le \sum_{a>T} \left(\frac{ k^2e}{ap_m} \right)^a \le \sum_{a>T}\left(\frac{1}{2}\right)^a = o(1).
\end{equation*}
Thus,
\begin{equation}\label{s2}
S_{> T} = o(N_T^kp_m^{(n-1) k}).
\end{equation}
So $S_{> T}$ is negligible. Therefore, from $(\ref{s1})$ and $(\ref{s2})$
\begin{equation*}
\mathbb{E}X_{n,m}^k = S_{\le T} + S_{> T} \le C_1^k N_T^kp_m^{(n-1) k}.
\end{equation*}
From $(\ref{EXn})$, we see that \begin{equation*}
(\mathbb{E}X_{n,m})^k = N_T^k p_m^{(n-1) k} \exp\left(- k\frac{1-p_m}{p_m} + O(kn^{-2/3}) \right) \ge C_2^k N_T^k p_m^{(n-1) k},
\end{equation*}
where $C_2 < \exp(-(1-p_m)/(p_m))$. Setting $C :=C_1 C_2^{-1}$, we obtain Lemma \ref{concentrationlemma}.
\end{proof}

\section{Law of the Iterated Logarithm}
Recall that $X_n$ is the number of spanning tree in $G(n, p)$. To prove Theorem \ref{maintheorem3}, for any $\varepsilon >0$ we need to show a lower bound

\begin{equation*}
\mathbb{P}\left[\frac{\log X_{n} - \mu_{n}}{\sigma} \geq (1-\varepsilon)\sqrt{2\log\log n} \text{ for infinitely many }n  \right]=1,
\end{equation*}
and an upper bound
\begin{equation*}
\mathbb{P}\left[ \frac{\log X_n - \mu_n}{\sigma} \geq (1+\varepsilon)\sqrt{2\log\log n} \text{ for infinitely many }n \right] = 0.
\end{equation*}

\begin{remark}
Our proof is by showing that  $\log X_n$ is tightly controlled by the number of edges in $G(n,p)$. We use results of Janson and standard technique for the lower bound, and we use Lemma \ref{XKBound} for the upper bound. As the number of edges is binomially distributed, $\log X_n$ inherits the LIL. 
\end{remark}

\subsection{Lower Bound}

To prove the lower bound of the LIL, we show that there exists a sequence $\{n_{k}\}_{k=1}^{\infty}$ such that for any fixed $\varepsilon > 0$,
\begin{equation*}
\mathbb{P}\left[\frac{\log X_{n_{k}} - \mu_{n_{k}}}{\sigma} \geq (1-\varepsilon)\sqrt{2\log\log n_{k}} \text{ for infinitely many }k  \right]=1.
\end{equation*}

Let $E_{n}$ be the random variable that counts the number of edges in $G(n,p)$, and let $E_{n}^* = (E_n - \mathbb{E}E_n)/\sqrt{\text{Var } E_n}$. Note that $E_n$ is a sum of iid's, so, from the proof of the law of the iterated logarithm in \cite[Chapter 10.2, Theorem 1]{chow2012probability},
 there is some sequence $\{n_k\} = \{a^k \}$ for some integer $a>1$ on which $E_{n_k}^*> (1-\varepsilon)\sqrt{2 \log \log \binom{n_k}{2}}$ infinitely often with probability $1$. Note that $\sqrt{2 \log \log \binom{n_k}{2}} \sim \sqrt{2 \log \log n_k}$, so $E_{n_k}^* > (1-\varepsilon)\sqrt{2 \log \log n_k}$ infinitely often with probability $1$.  From the proof of Theorem 6 in \cite{janson}, we have that
\begin{equation*}
\mathbb{P}\left[E_n^* - \frac{\log X_n - \mu_n}{\sigma} > C\right] = O(1/n)
\end{equation*}
for any positive constant $C$. 

Let $A_k$ be the event that $E_{n_k}^* -(\log X_{n_k} - \mu_{n_k})/\sigma > C$. By the choice of $\{n_k\}$, we have that \begin{equation*}
\sum_{k=1}^{\infty} \mathbb{P}[A_k] = \sum_{k=1}^{\infty} O(a^{-k}) < \infty.
\end{equation*}
So, by the Borel-Cantelli Lemma, $A_k$ holds for only finitely many $k$. Thus \begin{equation*}
E_{n_k}^* \le C + \frac{\log X_{n_k} - \mu_{n_k}}{\sigma}
\end{equation*}
holds for $k$ sufficiently large.

From the definition of $\{n_k\}$, we have that \begin{equation*}
\mathbb{P}\left[E_{n_k}^* > (1-\varepsilon/2)\sqrt{2\log \log n_k} \text{ infinitely often}\right] = 1.
\end{equation*}
Thus, with probability $1$, 
\begin{equation*}
C + \frac{\log X_{n_k} - \mu_{n_k}}{\sigma} > (1-\varepsilon/2)\sqrt{2\log \log n_k}
\end{equation*}
for infinitely many $k$. Since  $(\varepsilon/2) \sqrt{2 \log \log n}>C $ for $n$ sufficiently large, this gives the lower bound of the LIL.

\subsection{Upper Bound}

Fix $\varepsilon > 0$. 
By Lemma \ref{XKBound}, there exists a constant $K$ such that
\begin{equation*}
\mathbb{P}\left[ X_{n,m}\le K\mathbb{E}X_{n,m} \right]\ge1-n^{-4}.
\end{equation*}
Taking logarithms, we have $\log X_{n,m} \le \log \mathbb{E} X_{n,m} + \log K$ with probability at least $1-n^{-4}$.
By equation (\ref{EXn}),
\[
\log \mathbb{E} X_{n,m} =  \log N_T + (n-1) \log p_m  + O(1).
\]
Conditioning on $E_n = m$ in $G(n,p)$ and using the union bound over $\frac{p}{2} n^2 \le m \le \frac{1 + p}{2} n^2$, we have that with probability at least $1-n^{-2}$

\begin{equation}
1_{\mathcal{E}} [\log X_n] \le 1_{\mathcal{E}}\left(\log N_T + (n-1) \log \frac{E_n}{\binom{n}{2}} + O(1)\right),
\end{equation}
where ${\mathcal{E}}$ is the event that $G(n,p)$ has at least $\frac{p}{2} n^2$ edges and at most $\frac{1+p}{2} n^2$ edges, and $1_{\mathcal{E}}$ is the indicator random variable for $\mathcal{E}$. By Chernoff's bound, $1_{\mathcal{E}} =1 $ with probability at least $1 - n^2$, so 
\begin{equation}\label{XnEnBound}
\log X_n \le \log N_T + (n-1) \log \frac{E_n}{\binom{n}{2}} + O(1)
\end{equation}
holds with probability at least $1 - 2n^2$.

Note that $E_n = \text{Bin}(\binom{n}{2}, p)$ is a binomial distribution, so $\text{Var }{E_n} = \binom{n}{2}p(1-p)$ and $\mathbb{E}E_n = \binom{n}{2}p$. Expanding in terms of $E_n^*$, the normalized version of $E_n$, we have that

\begin{align*}
\log \frac{E_n}{\binom{n}{2}} &= \log \left( \frac{\sqrt{\operatorname{Var} E_n} E_n^*}{\binom{n}{2}} + \frac{ \mathbb E E_n}{\binom{n}{2}} \right) \\
&= \log \left( \left( \frac{2p(1-p)}{n(n-1)} \right)^{1/2} E_n^* + p \right) \\
&= \log p + \log \left( 1 + \sqrt{\frac{2(1-p)}{p}} \frac{E_n^*}{\sqrt{n(n-1)}} \right) \\
&= \log p + \sqrt{\frac{2(1-p)}{p}} \frac{E_n^*}{n-1} + O(1/n^2),
\end{align*}
where we take the Taylor expansion to get the last equality. Therefore, with probability at least $1-2n^{-2}$,

\begin{align*}
\log X_n &\le \log N_T + (n-1) \log p + \sqrt{\frac{2(1-p)}{ p}} E_n^* + O(1) = \mu_n + \sigma E_n^* + O(1). \end{align*}
Thus,
\begin{equation} \label{UpperBound}
\frac{\log X_n - \mu_n}{\sigma} \le E_n^* + O(1) 
\end{equation}
holds with probability at least $1 - n^2$.

Since $\sum_{n=1}^{\infty} n^{-2}$ is finite, by the Borel-Cantelli lemma, the event that \eqref{UpperBound} holds for all sufficiently large $n$ happens with probability $1$.  Since $E_n^*$ is the sum of $\binom{n}{2}$ iid random variables, we can use the LIL to conclude that, with probability $1$

\begin{align*}
E_n^* &\le (1+\varepsilon/2)\sqrt{2\log \log {n\choose2}} \\
&\le (1+\varepsilon/2)(\sqrt{2\log \log n} + \sqrt{2}) \\
&= (1+\varepsilon/2)\sqrt{2\log \log n} + O(1)
\end{align*}
holds for sufficiently large $n$. Taking $n$ large enough, we see that

\begin{equation*}
\frac{\log X_n - \mu_n}{\sigma} \le (1+\varepsilon)\sqrt{2\log \log n}
\end{equation*}
holds all but finitely many times with probability $1$. This completes the proof of the upper bound.

\section{Appendix}
\begin{proof}[Proof of equation (\ref{fallingfactorial})]
Let $N, \ell$ be such that $\ell = o(N^{2/3})$. Then,
\begin{equation*} \begin{split}
(N)_{\ell} & = N(N-1) \dotsb (N - \ell + 1) \\
& = N^{\ell} \prod_{i=0}^{\ell-1} (1-i/N) \\
& = N^{\ell} \prod_{i=1}^{\ell -1} \exp\left(-i/N + O(i^2/N^2)\right) \\
& = N^{\ell} \exp\left(\sum_{i=0}^{\ell-1} -i/N + O(i^2/N^2) \right) \\
& = N^{\ell} \exp \left( -\frac{\ell(\ell-1)}{2N} + O(\ell^3/N^2)\right) 
\end{split}\end{equation*}
\end{proof}

\begin{proof}[Proof of equation (\ref{N2})]
We show that \begin{equation*}
\frac{kn/(n^{4\alpha})!}{( e^{-k(k-1)}(k(k-1))^a)/a!} = o(1/n).
\end{equation*}
We use the trivial bounds \begin{equation*}
(n/e)^n < n! < n^n.
\end{equation*}
Then we compute
\begin{equation*} \begin{split}
\frac{(k n)/(n^{4\alpha})!}{( e^{-k(k-1)}(k(k-1))^a)/a!} 
& \le \frac{a! n e^{k(k-1)}}{(n^{4\alpha})!} \\
& \le \frac{(n^{3\alpha})! n^2e^{k^2}}{(n^{4\alpha})!} \\
& \le \frac{(n^{3\alpha})! n^2 e^{O(n^{2\alpha})}}{(n^{4\alpha})!} \\
& \le \frac{(n^{3\alpha})^{n^{3\alpha}} n^2 e^{O(n^{2\alpha})}}{(n^{4\alpha})!} \\
& \le \frac{(en^{3\alpha})^{n^{3\alpha}} }{(n^{4\alpha}/e)^{n^{4\alpha}}} \\
& \le \frac{e^{2n^{4\alpha}}}{n^{\alpha n^{4\alpha}}}\\
& = o(1/n).
\end{split} \end{equation*}
\end{proof}

\section{Acknowledgments}
The authors of the paper would like to show their gratitude to  Sam Payne for organizing SUMRY, the summer program where this research was done, to  Van Vu for his useful commentaries during the draft of this paper, and to  Michael Krivelevich for suggesting some improvements. This research was partially supported by NSF CAREER DMS-1149054.



\end{document}